\documentclass[times]{amsart}

\usepackage{amssymb,amsthm,amsmath}
\usepackage{fge}
\usepackage[left=1.5in,top=.9in,right=1.5in,bottom=1in]{geometry}
\usepackage{mathrsfs} 
\usepackage{bm}       
\usepackage{graphicx}
\usepackage[all]{xy}

\theoremstyle{plain}
\newtheorem{theorem}{Theorem}[section]
\newtheorem*{theorem*}{Theorem}
\newtheorem*{thm1}{Theorem 1}
\newtheorem*{thm2}{Theorem 2}

\newtheorem{proposition}[theorem]{Proposition}
\newtheorem{lemma}[theorem]{Lemma}
\newtheorem{corollary}[theorem]{Corollary}

\newtheorem{question}[theorem]{Question}

\theoremstyle{definition}
\newtheorem{definition}[theorem]{Definition}

\theoremstyle{remark}
\newtheorem{remark}[theorem]{Remark}

\newcommand{\sheaf}[1]{\mathscr{#1}}
\newcommand{\LL}{\sheaf{L}}
\newcommand{\OO}{\sheaf{O}}
\newcommand{\MM}{\sheaf{M}}
\newcommand{\EE}{\sheaf{E}}
\newcommand{\FF}{\sheaf{F}}
\newcommand{\NN}{\sheaf{N}}
\newcommand{\VV}{\sheaf{V}}
\newcommand{\WW}{\sheaf{W}}

\newcommand{\TT}{\sheaf{T}}
\newcommand{\GG}{\sheaf{G}}

\DeclareMathOperator{\HHom}{\sheaf{H}\!\mathit{om}}

\DeclareMathOperator{\End}{\mathrm{End}}

\DeclareMathOperator{\Pic}{\mathrm{Pic}}

\DeclareMathOperator{\SSpec}{\mathbf{Spec}}
\DeclareMathOperator{\Spec}{\mathrm{Spec}}
\DeclareMathOperator{\PProj}{\mathbf{Proj}}

\newcommand{\Group}[1]{\mathbf{#1}}
\newcommand{\GL}{\Group{GL}}

\newcommand{\GOrth}{\Group{GO}}
\newcommand{\GSOrth}{\Group{GSO}}
\newcommand{\SGOrth}{\GSOrth}
\newcommand{\muu}{\bm{\mu}}
\newcommand{\GSp}{\Group{GSp}}

\DeclareMathOperator{\IIsom}{\Group{Isom}}

\DeclareMathOperator{\SSim}{\Group{Sim}}
\DeclareMathOperator{\SSSim}{\Group{Sim}^+}

\DeclareMathOperator{\PPP}{\Group{P}}

\newcommand{\Dynkin}[1]{\mathsf{#1}}
\newcommand{\stack}[1]{\mathsf{#1}}

\newcommand{\R}{\mathbb R}
\newcommand{\A}{\mathbb A}

\renewcommand{\P}{\mathbb P}
\newcommand{\Gm}{\Group{G}_{\mathrm{m}}}

\newcommand{\mapto}[1]{\xrightarrow{#1}}

\newcommand{\subsetto}{\hookrightarrow}
\newcommand{\inv}{^{-1}}

\newcommand{\dual}{^{\vee}}

\newcommand{\pullback}{^{*}}
\newcommand{\pushforward}{_{*}}

\newcommand{\duality}{^{\sharp}}
\newcommand{\exterior}{{\textstyle \bigwedge}}
\newcommand{\tensor}{\otimes}
\newcommand{\adj}[1]{\psi_{#1}}

\DeclareMathOperator{\disc}{\mathrm{disc}}
\DeclareMathOperator{\rk}{\mathrm{rk}}
\newcommand{\can}{\mathrm{can}}
\newcommand{\ev}{\mathrm{ev}}
\newcommand{\id}{\mathrm{id}}
\newcommand{\vp}{\varphi}

\newcommand{\sig}{\sigma}

\newcommand{\et}{\mathrm{\acute{e}t}}

\newcommand{\isom}{\cong}

\newcommand{\Het}{H_{\et}}

\newcommand{\discs}{d}

\newcommand{\aff}[1]{\mathbb{V}(#1)}

\newcommand{\category}[1]{\mathsf{#1}}
\renewcommand{\stack}[1]{\bm{\mathsf{#1}}}
\newcommand{\VB}{\category{VB}}
\newcommand{\BF}{\category{BF}}

\newcommand{\VVB}{\stack{VB}}
\newcommand{\BBF}{\stack{BF}}

\newcommand{\PPic}{\stack{Pic}}

\newcommand{\sheafsharp}{^{\sharp}}
\newcommand{\rad}{\mathrm{rad}}

\newcommand{\exter}[2]{\exterior^{#1}{#2}}
\newcommand{\LGrass}{\Lambda\mathbb{G}}

\newcommand{\lpow}[2]{\lambda^{#1}{#2}}
\newcommand{\midext}[1]{\lambda(#1)}
\newcommand{\euler}[1]{\ensuremath{e}(#1)}

\DeclareMathOperator*{\bigperp}{\mbox{\Large$\bot$}}

\newcommand{\linedef}[1]{\textsl{#1}}

\usepackage{hyperref}
\hypersetup{pdftitle={Vector bundles of rank 4 and A_3=D_3},pdfauthor={Asher Auel}} 
\hypersetup{colorlinks=true,linkcolor=blue,anchorcolor=blue,citecolor=blue}

\begin{document}

\title{Vector bundles of rank four and $\Dynkin{A}_3=\Dynkin{D}_3$}

\author{Asher Auel}
\address{Emory University, Department of Mathematics \& CS \\ 400 Dowman Drive NE, Atlanta, GA 30322}

\email{auel@mathcs.emory.edu}



\begin{abstract}
Over a scheme $X$ with 2 invertible, we show that a version of
``Pascal's rule'' for vector bundles of rank 4 gives an explicit
isomorphism between the moduli functors represented by projective
homogeneous bundles for reductive group schemes of type $\Dynkin{A}_3$
and $\Dynkin{D}_3$.  We exploit this to prove that a vector bundle
$\VV$ of rank 4 has a sub- or quotient line bundle if and only if the
canonical symmetric bilinear form on $\exter{2}{\VV}$ has a lagrangian
subspace.  Under additional hypotheses on $X$, we prove that this is
equivalent to the vanishing of the Witt-theoretic Euler class
$\euler{\VV} \in W^4(X,\det\VV\dual)$.
\end{abstract}

\maketitle

\section{Introduction}

\label{sec:introduction}

Let $X$ be a scheme with 2 invertible and $\VV$ a vector bundle (i.e.\ a
locally free $\OO_X$-module) of even rank $r=2s$ on $X$.  The middle
exterior power $\exter{s}{\VV}$ supports a canonical regular
$\det(\VV)$-valued $(-1)^s$-symmetric bilinear form
$$
\exter{s}{\VV} \tensor \exter{s}{\VV} \mapto{\wedge} \det{\VV}
$$
given by wedging.  This operation defines a functor from the category
(resp.\ stack) of vector bundles of even rank to the category (resp.\
stack) of regular line bundle-valued bilinear forms (see Proposition \ref{prop:midext_functor}).  When $\VV$ is
free, the middle exterior power form on $\exter{s}{\VV}$ is hyperbolic.  On the other
hand, for example when $X=\P^4$ and $\VV=\Omega^1_{\P^4}$ (see
Walter~\cite{walter:GW_groups_projective_bundles}), the middle
exterior power form need not even be metabolic.  In this work, we give
a general necessary condition for the middle exterior form to be
metabolic, namely, that $\VV$ has a sub- or quotient line bundle (which
is locally a direct summand).  This is related to a version of
``Pascal's rule'' (see Lemma~\ref{lem:pascal}) for vector bundles.
When $\VV$ has rank 4, our first main result (see
Corollary~\ref{cor:lineiffmetabolic}) is that this condition is
sufficient as well:

\begin{thm1}
Let $X$ be a scheme with 2 invertible and $\VV$ be a vector bundle of
rank 4 on $X$.  Then $\VV$ has a sub- or quotient line bundle if and
only if the middle exterior power form on $\exter{2}{\VV}$ is metabolic.
\end{thm1}

The proof involves the interplay between an explicit moduli interpretation
and the (quite classical) geometry of the symmetric lagrangian
grassmannian (see \S\ref{sec:middle}) as well as the exceptional
isomorphism $\Dynkin{A}_3=\Dynkin{D}_3$ (see
Theorem~\ref{thm:LGrass_D3}).  We then give (see
Corollary~\ref{cor:with_witt_cancellation}) an interpretation of this
result involving the Euler class (see \S\ref{sec:Euler_classes}) in Witt
theory (in the sense of
Fasel--Srinivas~\cite{fasel_srinivas:chow-witt}).  

\begin{thm2}
Let $X$ be a scheme with 2 invertible satisfying the ``stable
metabolicity'' property.  Let $\VV$ a vector bundle of rank 4 on $X$.
Then $\VV$ has a sub- or quotient line bundle if and only if the
Witt-theoretic Euler class $e(\VV) \in W^4(X,\det\VV\dual)$ vanishes.
\end{thm2}

This result 
follows from Theorem~1 by using explicit formulas for Euler classes
found in Fasel--Srinivas~\cite[\S2.4]{fasel_srinivas:chow-witt}.  The
``stable metabolicity'' property is explored in
\S\ref{subsec:cancellation}, where a number of examples are provided.

Finally, in \S\ref{subsec:Euler_classes}, we discuss necessary and
sufficient conditions for the vanishing of the Euler class in
Grothendieck--Witt (GW) theory in terms of the existence of free sub- or
quotient line bundles (i.e.\ nonvanishing (co)sections).  
We prove the analogue of Theorem~1 (see Corollary~\ref{cor:metabolic_2})
for vector bundles of rank 2, as well as partial results
(see Proposition~\ref{prop:vanishing_2} and \ref{prop:vanishing_4})
concerning Euler classes of vector bundles of rank 2 and 4.

When 2 is not invertible, one should instead consider
the natural \emph{quadratic} form $\exter{s}{\VV} \to \det\VV$,
see~\cite[II~Prop.~10.12]{book_of_involutions}.  Likewise, replacing
$\VV$ by an Azumaya algebra $\sheaf{A}$ of degree $r=2s$, one should
consider the \linedef{canonical involution} on $\lambda^s\sheaf{A}$,
see~\cite[Ch.~II,~\S10.B]{book_of_involutions} and
Parimala--Sridharan~\cite{parimala_sridharan:norms_and_pfaffians}.
The generalization to these contexts of the results presented here
will be considered in the future.

This work was carried out while the author was a guest at the Max
Planck Institute for Mathematics in Bonn, Germany, where he enjoyed
excellent working conditions and many fruitful interactions.  In
particular, he would like to thank J.~Fasel for useful conversations
and for his encouragement of this project.  This work was partially
supported by the National Science Foundation Mathematical Sciences
Postdoctoral Research Fellowship DMS-0903039.

\section{Middle exterior power forms}
\label{sec:middle}

Let $X$ be a scheme with 2 invertible and $\LL$ a line bundle (i.e.\
an invertible $\OO_X$-module).  An \linedef{$\LL$-valued bilinear
form} on $X$ is a triple $(\EE,b,\LL)$, where $\EE$ is a vector bundle
on $X$ (i.e.\ locally free $\OO_X$-module of constant finite rank) and $b :
\EE\tensor\EE \to \LL$ is an $\OO_X$-module morphism.  An $\LL$-valued
bilinear form is \linedef{symmetric} (resp.\ \linedef{alternating}) if
$b$ factors through the canonical epimorphism $\EE\tensor\EE \to
S^2\EE$ (resp.\ $\EE\tensor\EE \to \exterior^2\EE$).  As shorthand, a
symmetric form will be called $(+1)$-symmetric and an alternating form
$(-1)$-symmetric.  An $\LL$-valued bilinear form $(\EE,b,\LL)$ is
\linedef{regular} if the canonical adjoint $\adj{b} : \EE \to
\HHom(\EE,\LL)$ is an isomorphism of vector bundles on $X$.  We will
mostly dispense with the adjective ``$\LL$-valued''.

\begin{definition}
Let $\VV$ be a vector bundle of even rank $r=2s$ on $X$.  Denote by
$\lpow{s}{\VV} = (\exter{s}{\VV}, \wedge, \det\VV)$ the $\det
\VV$-valued \linedef{middle exterior power form},
$$
\exter{s}{\VV} \tensor \exter{s}{\VV} \mapto{\wedge} \exter{r}{\VV} = \det \VV,
$$ 
on $X$.  It is $(-1)^s$-symmetric, regular, and of rank $n :=
\binom{r}{s}$.  
\end{definition}

\subsection{Torsorial interpretation}
\label{subsec:Torsorial_interpretation}

The signed discriminant is defined for $\LL$-valued forms of even rank
(cf.~Parimala--Sridharan~\cite[\S
4]{parimala_sridharan:norms_and_pfaffians}).  A \linedef{discriminant
module} is a pair $(\NN,n)$ consisting of a line bundle $\NN$ and
an $\OO_X$-module isomorphism $n : \NN^{\tensor 2} \to \OO_X$.  The isomorphism
class of a discriminant module yields an element of $\Het^1(X,\muu_2)$
(cf.~Milne~\cite[III \S4]{milne:etale_cohomology} or
Knus~\cite[III.3]{knus:quadratic_hermitian_forms}).  Applying the
determinant functor to the adjoint morphism of a regular bilinear form
$(\EE,b,\LL)$ of even rank $n=2m$ yields an isomorphism
$$
\det \EE \mapto{\det \adj{b}} \det \HHom(\EE,\LL) \mapto{\can}
\HHom(\det\EE, \LL^{\tensor n}),
$$
of vector bundles on $X$ giving rise to a discriminant module $\disc b
: (\det\EE \tensor (\LL\dual)^{\tensor m})^{\tensor 2} \to \OO_X$.  As
usual, we modify $\disc b$ by factor of $(-1)^m$ to define the \linedef{signed
discriminant} module $\discs(\EE,b,\LL)$.  We denote by $\langle 1 \rangle$
the trivial discriminant module $\OO_X^{\tensor 2} \to \OO_X$, whose
isomorphism class
is the identity of the group $\Het^1(X,\muu_2)$.

The first fundamental property of a middle exterior power form
is that it has trivial signed discriminant.  

\begin{lemma}
\label{lem:disc_triv}
Let $\VV$ be a vector bundle of rank $r=2s$ on $X$.
The middle exterior power form $\lpow{s}{\VV}$ has trivial
discriminant.  Moreover, there is a canonical choice of isomorphism
$\zeta_{\VV}: \discs (\lpow{s}{\VV}) \to \langle 1 \rangle$.
\end{lemma}
\begin{proof}
When $\VV$ is free with basis $e_1,\dotsc,e_r$, define an
$\OO_X$-module isomorphism $\zeta
: \det
\exter{s}{\VV} \to \det{\VV}^{\tensor \frac{1}{2}\binom{r}{s}}$ by
$$
\bigwedge_{1 \leq i_1 < \dotsm < i_s \leq r} 
\bigl( e_{i_1}\wedge \dotsm \wedge e_{i_s} \bigr)
\quad \mapsto \quad 
\bigl( e_1 \wedge \dotsm \wedge e_r \bigr)^{\tensor \frac{1}{2}\binom{r}{s}}. 
$$
A standard computation 
shows that $\zeta$ is $\GL(\VV)$-equivariant, thus does not depend on the choice of basis.  Hence for a
general vector bundle $\VV$ or rank $r$, the map $\zeta$ patches over a Zariski open
covering of $X$ splitting $\VV$.  Finally, tensoring $\zeta$ by
$\bigl( \det\VV \dual \bigr)^{\tensor \frac{1}{2}\binom{r}{s}}$
 and scaling by $(-1)^m$  yields an $\OO_X$-module morphism $\zeta_{\VV} : \discs
(\lpow{s}{\VV}) \to \langle 1 \rangle$, which is seen to be an
isomorphism of discriminant modules by a local verification.
\end{proof}

A \linedef{similarity} between bilinear forms $(\EE,b,\LL)$ and
$(\EE',b',\LL')$ is a pair $(\vp,\lambda)$ consisting of isomorphisms
$\vp : \EE \to \EE'$ and $\lambda : \LL \to \LL'$ such that either of
the following (equivalent) diagrams,
\begin{equation}
\label{similarity}
\begin{split}
\xymatrix{
\EE\tensor\EE \ar[d]_{\vp\tensor\vp} \ar[r]^{b}  & \LL \ar[d]^{\lambda} \\
\EE'\tensor\EE' \ar[r]^{b'} & \LL' \\
}
\qquad
\xymatrix{
\EE \ar[d]_{\vp} \ar[r]^-{\adj{b}}    & \HHom(\EE,\LL) \\ 
\EE' \ar[r]^-{\adj{{b'}}} & \HHom(\EE',\LL') \ar[u]_{{\lambda\inv} \vp^{\vee\LL}}\\
}
\end{split}
\end{equation}
of $\OO_X$-modules commute, where $\lambda\inv\vp^{\vee\LL}(\psi) =
\lambda\inv \circ \psi \circ \vp$ on sections.  Note that the
commutativity of diagrams~\eqref{similarity} takes on the
familiar formula $b'(\vp(v), \vp(w)) = \lambda \circ b(v, w)$ on
sections.  A similarity transformation $(\vp,\lambda)$ is an
\linedef{isometry} if $\LL=\LL'$ and $\lambda$ is the identity map.

Denote by $\SSim(\EE,b,\LL)$ the presheaf, on the large \'etale site
$X_{\et}$, of similitudes of a regular $\LL$-valued $(\pm 1)$-symmetric
form $(\EE,b,\LL)$.  In fact, this is a sheaf and is representable by
a smooth affine reductive group scheme over $X$.  See
Demazure--Gabriel~\cite[II.1.2.6, III.5.2.3]{demazure_gabriel}).  Here
we consider reductive group schemes whose fibers are not necessarily
geometrically integral, in contrast to SGA~3~\cite[III XIX.2]{SGA3}.
When $b$ is symmetric (resp.\ alternating), this is the \linedef{orthogonal similarity group}
$\GOrth(\EE,b,\LL)$ (resp.\ 
\linedef{symplectic similarity group} $\GSp(\EE,b,\LL)$).  Even though
these sheaves of groups are representable by schemes over $X$, we will
still think of them as sheaves of groups on $X_{\et}$.

Any similarity $(\vp,\lambda)$ between bilinear forms of even rank induces an
isomorphism of signed discriminants $\discs(\vp,\lambda) :
\discs(\EE,b,\LL) \to \discs(\EE',b',\LL')$.  When 2 is invertible on
$X$, denote by
$\SSSim(\EE,b,\LL)$ the sheaf kernel of the induced homomorphism
$\SSim(\EE,b,\LL) \to \IIsom(\discs(\EE,b,\LL)) = \muu_2$.
If $b$ is alternating, then
$\SSSim(\EE,b,\LL) = \SSim(\EE,b,\LL)$; the proof in~\cite[III Prop.\
12.3]{book_of_involutions} can be adapted to $X_{\et}$.  If $b$
is symmetric (and 2 is invertible on $X$), then $\SSSim(\EE,b,\LL)$ is
called the group of \linedef{proper} similarities; it is represented
by a smooth
connected reductive group scheme.

An \linedef{oriented bilinear form} $(\EE,b,\LL,\zeta)$ is a bilinear
form of even rank together with an isomorphism $\zeta :
\disc(\EE,b,\LL) \to \langle 1 \rangle$ of discriminant modules.  In
particular, any oriented bilinear form is regular and has trivial
signed discriminant.  An \linedef{oriented similarity} between
oriented bilinear forms $(\EE,b,\LL,\zeta)$ and
$(\EE',b',\LL',\zeta')$ is a triple $(\vp,\lambda,\xi)$ where
$(\vp,\lambda)$ is a similarity and $\xi : \discs(\EE,b,\LL) \to
\discs(\EE',b',\LL')$ is an isomorphism of discriminant modules such
that $\zeta' \circ \xi = \zeta$.  When $X$ is connected, every regular
symmetric bilinear form with trivial discriminant has two oriented
similarity classes.

For any vector bundle $\FF$, any line bundle $\LL$, and any $\epsilon
\in \{\pm 1\}$, the \linedef{hyperbolic form} $H_{\LL}^{\epsilon}(\FF)$
has underlying vector bundle $\FF \oplus \HHom(\FF,\LL)$ and bilinear
form given by $(v,f)\tensor(w,g) \mapsto f(w) + \epsilon g(v)$ on
sections.  It is $\epsilon$-symmetric and regular.  A hyperbolic form
carries a canonical orientation $\zeta_{\FF} : \discs(H_{\LL}^{\epsilon}(\FF))
\to \langle 1 \rangle$. Every oriented $\epsilon$-symmetric bilinear
form of rank $n=2m$ is, locally on $X_{\et}$, oriented similar to
$H_{\OO_X}^{\epsilon}(\OO_X^m)$, see~\cite[Thm.~1.15]{auel:clifford}.
For each $n=2m$, put $\SSim_{m,m}^{\epsilon} = \SSim(H_{\OO_X}^{\epsilon}(\OO_X^m))$ and
$\SSim_{m,m}^{+,\epsilon} = \SSSim(H_{\OO_X}^{\epsilon}(\OO_X^m))$.

We now place the middle exterior power form into a functorial framework.
For $r \geq 1$, denote by $\VB_{r}(X)$ the groupoid of vector bundles
of rank $r$ under isomorphism (this is isomorphic to the groupoid of
$\GL_r$-torsors on $X_{\et}$).  For even $n=2m \geq 1$ and $\epsilon
\in \{\pm 1\}$, denote by $\BF_{n}^{+,\epsilon}(X)$ the groupoid whose
objects are oriented (line bundle-valued) $\epsilon$-symmetric
bilinear forms of rank $n$ on $X$ under oriented similarities (this is
isomorphic to the groupoid of $\SSim_{m,m}^{+}$-torsors over
$X_{\et}$).  Denote by $\VVB_{r}$ and $\BBF_{n}^{+,\epsilon}$ the
associated stacks over $X_{\et}$.  There are canonical cartesian
functors $\det : \VVB_{r} \to \PPic$ and $\mu : \BBF_{n}^{+,\epsilon}
\to \PPic$ defined by the determinant and value line bundle,
respectively, where $\PPic$ is the Picard stack on $X_{\et}$.

\begin{proposition}
\label{prop:midext_functor}
Let $X$ be a scheme with 2 invertible, $r=2s$ even, and $n =
\binom{r}{s}$.  Then the middle exterior power induces a cartesian
functor $ \lpow{s}{} : \VVB_{r} \to \BBF_{n}^{+,(-1)^s}$ making the
following diagram
$$
\xymatrix@C=56pt{
\VVB_{r} \ar[d]_{\det} \ar[r]^(.45){\lpow{s}{}} & \BBF_{n}^{+,(-1)^s} \ar[d]^{\mu} \\
\PPic \ar@{=}[r] & \PPic
}
$$
of stacks over $X_{\et}$ commute.  In particular for each object $\VV$
of $\VB_{r}(X)$, there's a canonical homomorphism of sheaves of groups
$\lpow{s}{} : \GL(\VV) \to \SSSim(\lpow{s}{\VV})$ on $X_{\et}$, making
the following diagram
$$
\xymatrix@C=44pt{
\GL(\VV) \ar[d]_{\det} \ar[r]^(.45){\lpow{s}{}} & \SSSim(\lpow{s}{\VV}) \ar[d]^{\mu} \\
\Gm      \ar@{=}[r]       & \Gm               \\
}
$$
of sheaves of groups on $X_{\et}$ commute.
\end{proposition}
\begin{proof}
For each $U \to X$ in $X_{\et}$ define $\lpow{s} : \VB_r(U) \to
\BF_{n}^{+,(-1)^s}(U)$ on objects by sending $\VV$ to
$(\lpow{s}{\VV},\zeta_{\VV})$, where $\zeta_{\VV}$ is the orientation
defined in Lemma~\ref{lem:disc_triv}.  On morphisms, send an
$\OO_X$-isomorphism $\psi : \VV \to \VV'$ to the oriented similarity
$\Bigl(\exter{s}{\psi},\exter{r}{\psi},
\exter{n}{\bigl(\exter{s}{\psi}\bigr)}\tensor\bigl(\exter{r}{\psi\dual}\bigr)^{\tensor{\frac{1}{2}\binom{r}{s}}}\Bigr)$.
The only nontrivial thing to check is that this is indeed an oriented
similarity!  The fact that this gives rise to a cartesian functor
of stacks follows from its nature as a tensorial construction.  The
fact that the diagrams commute follows directly from the definition.
\end{proof}

\subsection{Some linear algebra}

Let $X$ be a scheme and $\VV \mapto{f} \NN$ be a morphism of $\OO_X$-modules.  
For $s \geq 1$, the map $\VV^{\tensor s} \to \NN \tensor \VV^{\tensor
(s-1)}$ defined on sections by
$$
v_0 \tensor \dotsm \tensor v_{s-1} \mapsto \sum_{j=0}^{s-1}
f(v_j)\tensor v_0 \tensor \dotsm \tensor \hat{v}_j \tensor
\dotsm \tensor v_{s-1} 
$$
descends to the \linedef{contraction} map $d^{(s-1)}f : \exter{s}{\VV} \to \NN
\tensor \exter{s-1}{\VV}$.  As usual, define $\exter{0}{\VV}=\OO_X$ so
that $d^{(0)}f=f$.
The composition
\begin{equation}
\label{eq:df_derivation}
\exter{s}{\VV} \mapto{d^{(s-1)}f} \NN \tensor \exter{s-1}{\VV}
\mapto{\id_{\NN} \tensor d^{(s-2)}f} 
\NN^{\tensor 2}\tensor \exter{s-2}{\VV}
\end{equation}
is zero.  
Hence for $r \geq 1$ we have an $r$-truncated
\linedef{Koszul complex}
$$
0 \to \exter{r}{\VV} \mapto{d^{(r-1)}f} \NN \tensor \exter{r-1}{\VV}
\to
\dotsm \to \NN^{\tensor(r-2)} \tensor \exter{2}{\VV}
\mapto{}
\NN^{\tensor (r-1)}\tensor {\VV} \mapto{\id_{\NN^{\tensor (r-1)}}\tensor f} \NN^{\tensor r} \to 0
$$
of $\OO_X$-modules.  If $\VV$ is a vector bundle of rank $r$ and $\NN$ is
a line bundle then the $r$-truncated
Koszul complex is denoted by $K(\VV,f)$; it is exact if $f$ is an epimorphism. 
Denoting by $\WW = \ker f$ and $j : \WW \to \VV$ the inclusion, we have a complex
\begin{equation}
\label{eq:almost}
0 \to \exter{s}{\WW} \mapto{\wedge^s j} \exter{s}{\VV} \mapto{d^{(s-1)}f} \NN
\tensor \exter{s-1}{\VV}
\end{equation}
of $\OO_X$-modules, which is exact under the conditions that $\VV$ is
a vector bundle, $\NN$ is a line bundle, and $f$ is an epimorphism.
(Such an $\NN$ will be called a \linedef{quotient line bundle} of
$\VV$.)  The following linear algebra lemma is well-known.

\begin{lemma}[Pascal's rule for vector bundles]
\label{lem:pascal}
Let 
$$
0 \to \WW \to \VV \mapto{f} \NN \to 0,
$$ 
be an exact sequence of vector bundles on $X$ with $\NN$ a line
bundle.  Then for any $s \geq 1$, the contraction map induces a short
exact sequence,
\begin{equation}
\label{eq:pascal}
0 \to \exter{s}{\WW} \to \exter{s}{\VV} \mapto{d^{(s-1)}f} \NN \tensor
\exter{s-1}{\WW} \to 0,
\end{equation}
of vector bundles on $X$.
\end{lemma}
\begin{proof}
We only need to verify that the image of $d^{(s-1)}f$ in
$\NN\tensor\exter{s-1}{\VV}$ is isomorphic to
$\NN\tensor\exter{s-1}{\WW}$.  Under our hypotheses, the
Koszul complex (in particular, the complex \eqref{eq:df_derivation})
is exact, hence the image of $d^{(s-1)}f$
is isomorphic to the kernel of $\id_{\NN}\tensor
d^{(s-2)}f : \NN\tensor\exter{s-1}{\VV} \to \NN^{\tensor 2}\tensor
\exter{s-2}{\VV}$.  Tensoring the degree $s-1$ version of \eqref{eq:almost} by $\NN$, the
kernel of $\id_{\NN}\tensor
d^{(s-2)}f$ is seen to be isomorphic to $\NN\tensor\exter{s-1}{\WW}$, proving the lemma.
\end{proof}

Similarly, if $0 \to \NN' \to \VV \mapto{g} \WW' \to 0$ is an exact
sequence of vector bundles on $X$ with $\NN'$ a line bundle (such an $\NN'$ will be call a \linedef{sub-line bundle} of $\VV$), then there
is an induced short exact sequence
\begin{equation}
\label{eq:pascal_dual}
0 \to \NN' \tensor \exter{s-1}{\WW'} \to \exter{s}{\VV}
\mapto{\wedge^s g} \exter{s}{\WW'} \to 0,
\end{equation}
by applying Lemma~\ref{lem:pascal} to $\VV\dual$, then dualizing.

\begin{remark}
\label{rem:serre_splitting}
A trivial sub-line bundle $\OO_X$ of $\VV$ is also called a
\linedef{nonvanishing global section}.  If $X$ is any noetherian
affine scheme of (Krull) dimension $d$ then Serre's splitting theorem
guarantees a nonvanishing global section if $\VV$ has rank $> d$.  If
$X$ is a variety of dimension $d$ over a field and $\VV$ has rank $>
d$, then $\VV$ has a nonvanishing global section as soon as it is
generated by global sections (cf.~\cite[II Exer.\
8.2]{hartshorne:algebraic_geometry}).  Furthermore, if $X$ is
projective over a field, then by a theorem of Serre, some twist
$\VV(n)$ is always generated by global sections (cf.~\cite[II Thm.\
5.17]{hartshorne:algebraic_geometry}).  Thus for a projective
variety $X$ of dimension $d$ over a field, any vector bundle $\VV$
of rank $> d$ on $X$ has a sub-line bundle
of the form $\OO_X(-n)$ for $n \gg 0$.
\end{remark}

A bilinear form $(\EE,b,\LL)$ of rank $n=2m$ on $X$ is
\linedef{metabolic} if there exists a vector subbundle $\FF \to \EE$
(i.e.\ locally a direct summand) of rank $m$ such that the restriction of $b$
to $\FF$ is zero.  Any choice of such $\FF$ is called a
\linedef{lagrangian}.  For example, $\FF$ is a lagrangian of the
hyperbolic form $H_{\LL}^{\epsilon}(\FF)$.  An $\OO_X$-submodule $\FF
\to \EE$ is a lagrangian if and only if $\FF$ is the kernel of
$\adj{b,\FF} : \EE \to \HHom(\FF,\LL)$ given by the composition of
$\adj{b} : \EE \to \HHom(\EE,\LL)$ and the projection $\HHom(\EE,\LL)
\to \HHom(\FF,\LL)$.  Thus a lagrangian $\FF$ is also equivalent to an
isomorphism of short exact sequences
$$
\xymatrix{
0 \ar[r] & \FF \ar[d]^{\adj{b}|_{\FF}} \ar[r] & \EE
\ar[d]_{\adj{b}} \ar[r] & \EE/\FF \ar[d]^{\adj{b,\FF}} \ar[r] & 0 \\
0 \ar[r] & \HHom(\HHom(\EE/\FF,\LL),\LL) \ar[r] & \HHom(\EE,\LL) \ar[r] & \HHom(\VV,\LL) \ar[r] & 0  
}
$$
of vector bundles, where the bottom sequence is the $\LL$-dual of the
top sequence.  We will often refer to a self-dual short exact
sequence, as above, by the term \linedef{lagrangian}.  See
Balmer~\cite{balmer:derived_witt_groups},
\cite{balmer:triangular_witt_groups_I}, and
\cite{balmer:triangular_witt_groups_II} for the general notion of
metabolic form in the context of Grothendieck(-Witt) groups of
triangulated and exact categories with duality.

\begin{proposition}
\label{prop:getting_metabolics}
Let $\VV$ be a vector bundle of rank $r=2s$ and $\NN$, $\NN'$ be line
bundles on $X$.
\begin{enumerate}
\item For an exact sequence of vector bundles of the form 
$$
0 \to \WW \to \VV \mapto{f} \NN \to 0,
$$ 
the associated exact sequence~\eqref{eq:pascal} is a lagrangian of
$\lpow{s}{\VV}$.

\item For an exact sequence of vector bundles of the form
$$
0 \to \NN' \to \VV \mapto{g} \WW' \to 0,
$$ 
the associated exact sequence~\eqref{eq:pascal_dual} is a lagrangian
of $\lpow{s}{\VV}$.

\item If $\VV \isom \NN \oplus \WW$, then $\lpow{s}{\VV}$ is
isometric to the hyperbolic form $H_{\det \VV}^{(-1)^s}({\bigwedge}^s
\WW)$.
\end{enumerate}
\end{proposition}
\begin{proof}
A linear algebra exercise using the description of lagrangians
as self-dual short exact sequences.
\end{proof}

\begin{corollary}
\label{cor:getting_metabolics}
If $\VV$ has a sub- or quotient line bundle then $\lpow{s}{\VV}$ is
metabolic.
\end{corollary}

The statement of Corollary~\ref{cor:getting_metabolics} can be
also seen indirectly as a consequence of the Whitney formula 
for the Euler class in
Grothendieck--Witt theory together with explicit formulas 
for its
computation (cf.\ Proposition~\ref{prop:euler_properties}).  We will
examine in further detail the connection between middle exterior power forms and the Euler class in \S\ref{sec:Euler_classes}.

As a consequence, we give a generalization to line bundle-valued forms
over schemes of a result of
Knus~\cite[V~Prop.~5.1.10]{knus:quadratic_hermitian_forms}.  
A line bundle-valued quadratic form of even rank and trivial
discriminant has an associated \linedef{refined Clifford invariant} in
$\Het^2(X,\muu_2)$ constructed in \cite[\S2.8]{auel:clifford} (denoted
by $gc^+$).  Two bilinear forms are \linedef{projectively similar} if
they are similar up to tensoring by a regular line bundle-valued
bilinear form of rank~1.

\begin{corollary}
\label{cor:knus}
Let $X$ be a scheme with 2 invertible and with the property that any
vector bundle of rank 4 has a sub- or quotient line bundle.  Then any
line bundle-valued symmetric bilinear form of rank 6 over $X$ with
trivial discriminant and
trivial refined Clifford invariant is
metabolic.
\end{corollary}
\begin{proof}
Any regular line bundle-valued symmetric bilinear form $(\EE,q,\LL)$
of rank 6 with trivial discriminant is isometric to a reduced pfaffian
form (see Bichsel--Knus~\cite[Thm.\
5.2]{bichsel_knus:values_line_bundles}).  Any reduced pfaffian form of
rank 6 with trivial refined Clifford invariant is projectively similar
to $\lpow{2}{\VV}$ for some $\VV$ of rank 4 (see
\cite[Prop.~2.21]{auel:clifford}).  The hypotheses and
Corollary~\ref{cor:getting_metabolics} imply that $\lpow{2}{\VV}$ is
metabolic, hence $(\EE,q,\LL)$ is metabolic.
\end{proof}

By Remark~\ref{rem:serre_splitting}, any $X$ of dimension $\leq 3$
that is either noetherian affine or projective over a field satisfies
the hypotheses of Corollary~\ref{cor:knus}. 

\begin{remark}
For a general vector bundle $\VV$, let $p : \P(\VV) = \PProj S(\VV)
\to X$ be the associated projective bundle.  Due to the universal
quotient line bundle $p\pullback\VV \to \OO_{\P(\VV)}(1)$, the form
$p\pullback (\lpow{s}{\VV})$ is always metabolic, with a canonical
choice of lagrangian given by Proposition~\ref{prop:getting_metabolics}.
Thus any section $\sig : X \to \P(\VV)$ of $p$ (i.e.\ $\sig \in
\P(\VV)(X)$) 
gives rise to a lagrangian of $\lpow{s}{\VV} = \sig\pullback
p\pullback (\lpow{s}{\VV})$ via pullback.
But each $\sig \in \P(\VV)(X)$ corresponds to
a quotient line bundle $\VV = \sig\pullback p\pullback \VV \to
\sig\pullback \OO_{\P(\VV)}(1)$, 
providing another way
of understanding Proposition~\ref{prop:getting_metabolics}, see \S\ref{subsec:Lagrangian_grassmannians} for
more details. In
particular, $\lpow{s}{\VV}$ is metabolic if $\P(\VV)(X) \neq
\varnothing$.  
\end{remark}

\subsection{Lagrangian grassmannian}
\label{subsec:Lagrangian_grassmannians}

The lagrangian grassmannian (or grassmannian of maximal isotropic
subspaces) of a bilinear form is a well-studied object.  Our
perspective is to consider these spaces as moduli functors and as
projective homogeneous space fibrations.  

A \linedef{polarization} on a scheme $X$ is an isomorphism class of
$\OO_X$-module $\FF$.  A morphism $(f,\vp) : (X,\FF) \to (Y,\GG)$ of
polarized schemes is a morphism $f : X \to Y$ of schemes together with
a morphism of $\OO_Y$-modules $\vp : \GG \to f\pushforward\FF$.

Let $(\EE,b,\LL)$ be a regular $\epsilon$-symmetric bilinear form of
even rank $n=2m$ on $X$.  Denote by $\LGrass(\EE,b,\LL)$ its moduli
space of lagrangians, called the \linedef{lagrangian grassmannian},
and $\TT$ the universal lagrangian subbundle.  The polarized $X$-scheme
$(\LGrass(\EE,b,\LL),\TT)$ represents the functor
$$
u : U \to X \quad \mapsto \quad \bigl\{ \text{lagrangians $\FF \mapto{\vp} u\pullback\EE$ of $u\pullback(\EE,b,\LL)$} \bigr\}
$$
where $\FF \mapto{\vp} u\pullback\EE$ and $\FF' \mapto{\vp'}
u\pullback\EE$ are considered equivalent if and only if there exists
an $\OO_U$-module isomorphism $\psi : \FF \to \FF'$ such that $\vp =
\vp' \circ \psi$.  Note that if $(\EE,b,\LL)$ is not a split
metabolic, then an isomorphism $\FF \to \FF'$ of lagrangians cannot
necessarily be extended to an isometry of $(\EE,b,\LL)$.

The stabilizer subgroup $\PPP_{\FF} = \PPP_{\FF}(\EE,b,\LL) \to
\SSim(\EE,b,\LL)$ of any lagrangian $\FF \to \EE$ is a parabolic
subgroup.
The morphism $\LGrass(\EE,b,\LL) \to X$ is a projective homogeneous
space for the group scheme $\SSim(\EE,b,\LL)$, i.e.\ a moduli space of
parabolic subgroups of $\SSim(\EE,b,\LL)$ of a given type.  It has
geometrically integral fibers when $\epsilon=-1$ (the $\Dynkin{C}_m$
case).  When $\epsilon=1$ (the $\Dynkin{D}_m$ case), it's Stein
factorization is of the form $\LGrass(\EE,b,\LL) \mapto{\pi} Z
\mapto{f} X$ where $\pi$ is smooth projective and with geometrically
integral fibers and $f$ is the \'etale double cover of $X$ given by
the signed discriminant $\discs(\EE,b,\LL)$.  Indeed, by going to
fibers over $X$, this is a consequence of the corresponding fact over
fields; see \cite[\S85]{elman_karpenko_merkurjev}) or
\cite[Prop.~3.3]{hassett_varilly:K3} for a more global argument.
When the signed discriminant is trivial, $\LGrass(\EE,b,\LL)$ has two
connected components over $X$, each of which is a $\SSSim(\EE,b,\LL)$ orbit.
In this case, two lagrangians $\FF \to \EE$ and $\FF' \to \EE$
correspond to points in the same connected component if and
only if $\dim_{\kappa(x)}(\FF({x}) \cap \FF'({x})) \equiv m \bmod 2$
for each point $x$ of $X$, where $\FF(x)$ is the fiber over the
residue field $\kappa(x)$, see Bourbaki \cite[\S5 Exer.\
18d]{bourbaki:quadratic}, Mumford
\cite{mumford:theta_characteristics}, or
Fulton~\cite{fulton:flag_bundles}.

Recall that for a vector bundle $\VV$ of rank $r$ on $X$, if
$\OO_{\P(\VV)/X}(1)$ is the universal quotient line bundle on $\P(\VV) \to X$, then the
polarized $X$-scheme $(\P(\VV),\OO_{\P(\VV)/X}(1))$ represents the moduli functor
$$
u : U \to X \quad \mapsto \quad \bigl\{ \text{quotient line bundles}~ u\pullback\VV \mapto{f}
\NN \bigr\}
$$
where $u\pullback\VV \mapto{f} \NN$ and $u\pullback\VV \mapto{f'}
\NN'$ are considered equivalent if and only if there exists an
isomorphism $\mu : \NN \to \NN'$ such that $f' = \mu \circ f$.  By
dualizing, we can view $\P(\VV\dual)$ as representing the moduli space
of sub-line bundles of $\VV$.

The projective bundle $\P(\VV) \to X$ is a projective homogeneous space for the
reductive group scheme $\GL(\VV)$ of type $\Dynkin{A}_{r-1}$.  The
stabilizer subgroup $\PPP_{\WW} = \PPP_{\WW}(\VV) \to \GL(\VV)$ of any
vector subbundle $\WW \to \VV$ with line bundle quotient is a parabolic
subgroup.

\begin{theorem}
\label{thm:map_to_LGrass}
Let $X$ be a scheme with 2 invertible and $\VV$ be a vector bundle of
even rank $r=2s$ on $X$.  
There exist canonical $X$-morphisms of schemes
$$
\Phi_{\VV}  : \P(\VV) \to \LGrass(\lpow{s}{\VV}), \qquad
\Phi_{\VV}'  : \P(\VV\dual) \to \LGrass(\lpow{s}{\VV}),
$$
which realize, on the level of moduli functors, the maps
\begin{align*}
\bigl( 0 \to \WW \to \VV \mapto{f} \NN \to 0 \bigr) \quad & \mapsto
\quad \bigl( 0 \to \exter{s}{\WW} \to \exter{s}{\VV} \mapto{df}
\NN \tensor \exter{s-1}{\WW} \to 0 \bigr) \\
\bigl( 0 \to \NN' \to \VV \mapto{g} \WW' \to 0 \bigr) \quad & \mapsto
\quad \bigl( 0 \to \NN' \tensor \exter{s-1}{\WW'} \to
\exter{s}{\VV} \mapto{\wedge^s g} \exter{s}{\WW'} \to 0 \bigr)
\end{align*}
defined by
Proposition~\ref{prop:getting_metabolics}, 
respectively (here $\NN$ and $\NN'$ are line bundles).  Furthermore, $\Phi_{\VV}$ and
$\Phi_{\VV}'$ map to different connected components over $X$ if and only if $s$
is a power of 2.  
\end{theorem}
\begin{proof}
Since the Koszul complex is functorial, the maps defined by
Proposition~\ref{prop:getting_metabolics} define morphisms of moduli functors.
By Yoneda's lemma, these are
representable by $X$-morphism of the representing moduli schemes.

Now we prove the final claim.  If $s$ is odd (the $\Dynkin{C}_m$ case)
then $\LGrass(\lpow{s}{\VV}) \to X$ is connected, so we need only
consider the case when $s$ is even (the $\Dynkin{D}_m$ case). The
fibral criterion for two lagrangians to be in the same connected
component reduces us to linear algebra considerations.  Let $V$
be a vector space of dimension $r=2s$ (with $s$ even) over a field
$k$, let $W, W' \subset V$ be subspaces of codimension 1, and choose a
splitting $V = W' \oplus L'$.  
The
codimension of $W \cap W' \subset V$ is either 1 or 2.  In the former case,
$W=W'$ and we have that $\exter{s}{W} \cap L' \tensor \exter{s-1}{W'} = \{0\}$
since $W \cap L' = \{0\}$.  As the rank of an
intersection of lagrangians has a well-defined parity, we see that
$\dim_k\bigl(\exter{s}{W} \cap L' \tensor \exter{s-1}{W'}\bigr)$ is always even.
Finally, we use the fact (cf.\
\cite[II~Lemma~10.29]{book_of_involutions}) that $\dim_k (\lpow{s}{V}) = \binom{r}{s}
\equiv 2 \bmod 4$ if and only if $s$ is a power of 2.  
Thus
$\dim_k\bigl(\exter{s}{W} \cap L' \tensor \exter{s-1}{W'}\bigr)
\not\equiv \frac{1}{2}\binom{r}{s} \bmod 2$ if and only if $s$ is a
power of 2 and the claim follows by appealing to the above fibral criterion.
\end{proof}

\section{The Euler class in Grothendieck--Witt theory}
\label{sec:Euler_classes}

Let $X$ be a scheme with 2 invertible.  Modeled on classical
treatments of the Koszul complex (cf.~\cite[\S1.6]{bruns_herzog:cohen-macauley_rings}), Balmer, Gille, and
Nenashev~\cite{balmer_gille:koszul}, \cite{gille_nenashev:pairings}
introduce the \linedef{Euler class} $\euler{\VV} \in
GW^{r}(X,\det\VV\dual)$ of a vector bundle $\VV$ of rank $r$ on $X$.
Also see Fasel \cite[\S2.4]{fasel:degree}. 
Here we mostly follow the treatment in
Fasel--Srinivas~\cite[\S2.4]{fasel_srinivas:chow-witt}.  

For a scheme $X$ with 2 invertible and $\NN$ a line bundle on $X$, the
$r$-shifted $\NN$-twisted derived Grothendieck--Witt groups
$GW^r(X,\NN)$ were introduced by
Balmer~\cite{balmer:derived_witt_groups},
\cite{balmer:triangular_witt_groups_I},
\cite{balmer:triangular_witt_groups_II} and
Walter~\cite{walter:GW_groups_triangulated_categories}.  For $\epsilon
\in \{\pm 1\}$, the classical Grothendieck--Witt group
$GW^{\epsilon}(X,\NN)$ of regular $\epsilon$-symmetric bilinear forms
on $X$ was introduced by Knebusch~\cite{knebusch}.  The map taking a
regular $\epsilon$-symmetric bilinear form $(\EE,b,\NN)$ to the
complex $\EE[s]$ consisting of $\EE$ concentrated in degree $s$
together with the natural $2s$-shifted $\NN$-twisted symmetric
isomorphism induced by $b$ 
gives rise to a group
isomorphism $GW^{(-1)^s}(X,\NN) \to GW^{2s}(X,\NN)$ for any $s$,
see Walter~\cite[Thm.~6.1]{walter:GW_groups_triangulated_categories}.

\subsection{Euler classes and middle exterior power forms}
\label{subsec:General_properties}

Let $\VV$ be a vector bundle of rank $r$ and $\NN$ a line bundle on a
scheme $X$. Given an $\OO_X$-module morphism $\VV \mapto{f} \NN$, consider the (twisted) Koszul complex
$K(\VV,f)$
$$
0 \to \bigl(\NN\dual \bigr)^{\tensor (r-1)} \tensor \exter{r}{\VV} \to
\dotsm \to \NN\dual \tensor \exter{2}{\VV} \mapto{d^{(1)}f}
\VV \mapto{f} \NN \to 0
$$  
of vector bundles on $X$, where $\NN$ is in degree 0.  Then $\NN\dual
\tensor K(\VV,f)$ is isomorphic to a standard Koszul complex
$K(\NN\dual\tensor\VV,\ev \circ (\id\tensor f))$ associated to the
cosection $\NN\dual\tensor\VV \to \OO_X$.  
The perfect
pairings $\exter{j}{\VV}\tensor\exter{r-j}{\VV}\to\det\VV$ together
with certain sign conventions give rise to a symmetric isomorphism of
complexes $\Phi : K(\VV,f) \to K(\VV,f)\duality$, where $(-)\duality =
\HHom\bigl(-,(\NN\dual)^{\tensor (r-2)}\tensor\det\VV\bigr)[r]$, see
Fasel--Srinivas~\cite[\S2.4]{fasel_srinivas:chow-witt}
or Fasel \cite[\S2.4]{fasel:degree}.  Thus the pair
$(K(\VV,f),\Phi)$ gives rise to a class in $GW^r(X,(\NN\dual)^{\tensor
(r-2)}\tensor\det\VV)$.

Consider the associated affine bundle $p : \aff{\VV} = \SSpec\, S(\VV)
\to X$
and
its zero section $s: X \to \aff{\VV}$.  Then
$p\pullback\VV\dual$ has a canonical ``evaluation'' morphism
$f : p\pullback\VV\dual \to \OO_{\aff{\VV}}$ with cokernel
$s{}\pushforward \OO_X$.
There's an associated Koszul complex $K(p\pullback\VV\dual, f)$
$$
0 \to \exter{r}{p\pullback\VV\dual} \to \dotsm \to
\exter{2}{p\pullback\VV\dual} \mapto{d^{(1)}f} p\pullback\VV\dual
\mapto{f} \OO_{\aff{\VV}} \to 0
$$
of vector bundles on $\aff{\VV}$. 
The canonical morphism of sheaves $s\sheafsharp : \OO_{\aff{\VV}} \to
s{}\pushforward \OO_X$ defines an isomorphism $K(p\pullback\VV\dual,f)
\to s{}\pushforward\OO_X$ in the derived category.  The zero section
defines a pullback map
$$
s\pullback :
GW^r(\aff{\VV},p\pullback\!\det\VV\dual) \to GW^r(X,\det\VV\dual)
$$ 
on Grothendieck--Witt groups.

\begin{definition}
Let $\VV$ be a vector bundle of rank $r$ on $X$.  The \linedef{Euler
class} $\euler{\VV} \in GW^r(X,\det\VV\dual)$ is defined to be
$s\pullback (K(p\pullback\VV\dual,f),\Phi)$.  The \linedef{middle
exterior power class} $\midext{\VV} \in GW^r(X, \det\VV)$ is defined
to be the Grothendieck--Witt group class $\lpow{s}{\VV}[r]$ 
if $r=2s$ is even and to be 0 if
$r$ is odd.
\end{definition}

\begin{proposition}[{Fasel--Srinivas~\cite[Prop.~14, 21]{fasel_srinivas:chow-witt}}]
\label{prop:euler_properties}
Let $X$ be a scheme with 2 invertible and $\VV$ a vector bundle of
rank $r$ on $X$.
\begin{enumerate}
\item (Explicit formula for Euler classes) In $GW^r(X,\det\VV\dual)$,
we have the following formula
$$
\euler{\VV} = 
\langle(-1)^{s(s-1)/2}\rangle\tensor\midext{\VV\dual} + 
H_{\det\VV\dual} \left( \sum_{j=0}^{\lfloor ({s-1})/{2}\rfloor}
(-1)^j\bigl[\exter{j}{\VV\dual}\bigr] \right).
$$

\item (Whitney sum formula for Euler classes) If $0 \to \WW \to \VV
\to \NN \to 0$ is an exact sequence of vector bundles, then
$\euler{\VV} = \euler{\WW} \cdot \euler{\NN}$ under the multiplication
in Grothendieck--Witt groups induced by $\det\WW\dual \tensor
\det\NN\dual \to \det\VV\dual$.
\end{enumerate}
\end{proposition}

\begin{remark}
\label{rem:dual_or_not}
Tensoring with $\det\VV\dual$ --- actually multiplication in the
Grothendieck--Witt group by the rank one bilinear form $(\det\VV\dual, \tensor,
(\det\VV\dual)^{\tensor 2})$ --- yields a canonical
isomorphism of groups $GW^r(X,\det\VV) \to GW^r(X,\det\VV\dual)$ via
the evaluation morphism $\det\VV \tensor (\det\VV\dual)^{\tensor 2}
\to \det\VV\dual$.  Under this isomorphism the class of
$\lpow{s}{\VV}$ is mapped to the class of $\lpow{s}{\VV\dual}$.
Also note that $\lpow{s}{\VV}$ is metabolic if and only if $\langle
\pm 1 \rangle \tensor \lpow{s}{\VV}$ is metabolic.
\end{remark}

The following corollary can be viewed as a generalization of
Calm\`es--Hornbostel~\cite[Rem.~7.5]{calmes_hornbostel:push-forwards}.

\begin{corollary}
\label{cor:euler_metabolic}
If $\VV$ has a sub- or quotient line bundle then $\euler{\VV}$ is
metabolic.
\end{corollary}
\begin{proof}
Combine the explicit formula for Euler classes (cf.\ Proposition~\ref{prop:euler_properties}) and
Corollary~\ref{cor:getting_metabolics}.  Another proof employs the
Whitney sum formula for Euler classes (cf.\ Proposition~\ref{prop:euler_properties})
noting that the product of metabolic
classes is metabolic.
\end{proof}

\subsection{Euler classes and quotient line bundles}
\label{subsec:Euler_classes_and_quotient_line_bundles}

For any line bundle $\NN$, Karoubi periodicity gives rise to the exact
sequence
\begin{equation}
\label{eq:Karoubi_periodicity}
GW^{r-1}(X,\NN) \mapto{f} K_0(X) \mapto{H_{\NN}} GW^r(X,\NN) \to
W^r(X,\NN) \to 0
\end{equation}
where $f$ is the forgetful map, see Walter~\cite[Thm.\
2.6]{walter:GW_groups_triangulated_categories}.  If a vector bundle
$\VV$ of even rank $r=2s$ on $X$ has a sub- or quotient line bundle $\LL$,
then by Propositions~\ref{prop:getting_metabolics} and \ref{prop:euler_properties}, the Euler class
$\euler{\VV}$ is hyperbolic:
\begin{equation}
\label{eq:split_euler}
\euler{\VV} = H_{\det\VV\dual}\biggl(
(-1)^s\bigl[\exter{s}{\WW\dual}\bigr] + \sum_{j=0}^{s-1}(-1)^j \bigl[\exter{j}{\VV\dual}\bigr]
\biggr),
\end{equation}
where $\WW$ is the vector bundle complementary to $\LL$.
We can give a useful representation of this class.  For any vector
bundle $\WW$, denote by 
$$
\exter{}{\WW} = \sum_{j=0}^{\rk \WW} (-1)^j \bigl[\exter{j}{\WW\dual}\bigr]
$$
in $K_0(X)$.  Recall that $\exter{}{\WW} = s{}\pullback
s{}\pushforward \OO_X$, where $s : X \to \aff{\WW}$ is the zero
section and $s\pushforward$ and $s\pullback$ are the associated
pushforward and pullback maps on $K_0$.

\begin{proposition}
\label{prop:nice_rep}
Let $\VV$ be a vector bundle of rank $r$ on $X$ and suppose that $0
\to \WW \to \VV \to \LL \to 0$ is a short exact sequence of vector
bundles with $\LL$ a line bundle.  Then we have
$$
\euler{\VV} = H_{\det\VV\dual}\bigl(  
\exter{}{\WW} 
\bigr)
$$
in $GW^r(X,\det\VV\dual)$.
\end{proposition}
\begin{proof}
The proof is a straightforward calculation using the explicit formula
for Euler classes (cf.\ Proposition~\ref{prop:euler_properties}),
Lemma~\ref{lem:pascal}, and the fact that, for all $j \geq 1$,
$$
\bigl[\LL\dual\tensor\exter{j-1}{\WW\dual}\bigr] -
\bigl[\exter{r-j}{\WW\dual}\bigr] 
$$
is in the kernel of the hyperbolic map $H_{\det\VV\dual} : K_0(X) \to
GW^r(X,\det\VV\dual)$.  
\end{proof}

In the case when $\VV$ has odd rank, a formula similar to Proposition~\ref{prop:nice_rep} appears in Fasel~\cite[Thm.~10.1]{fasel:I-cohomology}.  
For future reference, we record the equality
\begin{equation}
\label{eq:dual}
\det\WW \tensor \exter{}{\WW} = (-1)^{\rk\WW}\exter{}{\WW\dual}
\end{equation}
in $K_0(X)$.

\section{Rank four vector bundles}
\label{sec:Euler_classes_of_rank_four_vector_bundles}

In this section, we apply the above general results to the specific
case of vector bundles of rank four.  In this special situation, we
are helped by the exceptional isomorphism
$\Dynkin{A}_3=\Dynkin{D}_3$.

\subsection{Middle exterior forms}
\label{subsec:Rank_four_vector_bundles}

Let $X$ be a scheme with 2 invertible. By
Proposition~\ref{prop:midext_functor}, the middle exterior power
functor gives rise to a canonical homomorphism $\lpow{2}{} : \GL(\VV)
\to \GSOrth(\lpow{2}{\VV})$.
When $\VV$ has rank four, this
homomorphism is an isogeny.

\begin{proposition}
\label{prop:A3=D3}
Let $X$ be a scheme with 2 invertible and $\VV$ be a vector bundle of
rank 4 on $X$. There's a short exact sequence
$$
1 \to \muu_2 \to \GL(\VV) \mapto{\lpow{2}{}} \GSOrth(\lpow{2}{\VV})
\to 1
$$ 
of sheaves of groups on $X_{\et}$.
\end{proposition}
\begin{proof}
The only thing to check is that $\lpow{2}{}$ is an epimorphism on
$X_{\et}$, which is known, cf.\ Knus~\cite[V.5.6]{knus:quadratic_hermitian_forms}.
\end{proof}

\begin{remark}
The coboundary map $\Het^1(X,\GSOrth(\lpow{2}{\VV})) \to
\Het^2(X,\muu_2)$ associated to the exact sequence in
Proposition~\ref{prop:A3=D3} has the following interpretation: an
(oriented) regular line bundle-valued symmetric bilinear form of rank
6 is mapped to its refined Clifford invariant constructed in \cite[\S2.8]{auel:clifford},
cf.\ Corollary~\ref{cor:knus}.
\end{remark}

The $\Dynkin{D}_3$ version of Theorem~\ref{thm:map_to_LGrass} is more
precise, yielding a stronger version of
Corollary~\ref{cor:getting_metabolics} as follows.

\begin{theorem}
\label{thm:LGrass_D3}
Let $X$ be a scheme with 2 invertible and $\VV$ be a vector bundle of
rank 4 on $X$.  Then
$$
\Phi_{\VV} \sqcup \Phi_{\VV}' : \P(\VV) \sqcup \P(\VV\dual) \to \LGrass(\lpow{2}{\VV})
$$ 
is an isomorphism of $X$-schemes.
\end{theorem}
\begin{proof}
One can argue, by passing to fibers, using the classical case over
fields. We prefer to argue directly as follows.  Let $\Group{P}
\subsetto \GL_4$ be the parabolic subgroup given by the stabilizer of
$\OO_X^3 \subset \OO_X^4$ and let $\Group{Q} \subsetto \SGOrth_{3,3}$
be the parabolic subgroup corresponding to the associated choice of
oriented lagrangian $\exter{2}{\OO_X^3}$ of $\lpow{2}{\OO_X^4}\isom
H_{\OO_X}(\OO_X^3)$.  Then upon restricting the morphism $\lpow{2}{}$,
we have a commutative diagram
$$
\xymatrix@C=46pt@R=18pt{
1 \ar[r] & \muu_2 \ar@{=}[d] \ar[r] & \Group{P} \ar@{^{(}->}[d]
\ar[r]^(.47){\lpow{2}{}} & \Group{Q} \ar@{^{(}->}[d] \ar[r] & 1 \\
1 \ar[r] & \muu_2 \ar[r] & \GL_4 \ar[r]^(.47){\lpow{2}{}}
&\GSOrth_{3,3} \ar[r] & 1
}
$$
of groups schemes on $X$ extending Proposition~\ref{prop:A3=D3}.
Similarly, we have a epimorphism of (right) torsors
$$
\xymatrix@C=46pt{
\IIsom(\OO_X^4,\VV)
\ar@{->>}[r]^(.45){\lpow{2}{}} & \SSSim(\lpow{2}{\OO_X^4},\lpow{2}{\VV})
}
$$
equivariant for the corresponding homomorphism of group schemes.
Then in this situation, we have an induced commutative diagram of $X$-scheme isomorphisms 
$$
\xymatrix@C=40pt@R=18pt{
\IIsom(\OO_X^4,\VV)/\Group{P}
\ar[d]_(.45){\text{\rotatebox{90}{$\sim$}}} \ar[r]^(.45){\sim} & 
\SSSim(\lpow{2}{\OO_X^4},\lpow{2}{\VV})/\Group{Q} \ar[d]^(.45){\text{\rotatebox{90}{$\sim$}}} \\
\P(\VV)
\ar[r]^(.45){\Phi_{\VV}} & 
\LGrass(\lpow{2}{\VV})^{\circ}
}
$$ 
where $\LGrass(\lpow{2}{\VV})^{\circ}$ is the connected component
containing the image of $\Phi_{\VV}$ and the vertical isomorphism are
a consequence of SGA~3~\cite[III XXVI.3 Lemma 3.2]{SGA3}, since these
projective homogeneous schemes are moduli spaces of parabolic
subgroups of the associated groups.

Note that if $\GL(\VV)$ actually had a parabolic subgroup
$\Group{P}_{\WW}$ of the same type as $\Group{P}$ (i.e.\ there exists
an exact sequence $0 \to \WW \to \VV \to \LL \to 0$ with $\LL$ a line
bundle), then the above argument can be summarized with the following
commutative diagram
$$
\xymatrix@C=40pt@R=18pt{
1 \ar[r] & \muu_2 \ar@{=}[d] \ar[r] & \Group{P}_{\WW} \ar@{^{(}->}[d]
\ar[r]^(.45){\lpow{2}{}} & \Group{P}_{\lpow{2}{\WW}} \ar@{^{(}->}[d] \ar[r] & 1 \\
1 \ar[r] & \muu_2 \ar[r] & \GL(\VV) \ar@{->>}[d] \ar[r]^(.45){\lpow{2}{}} &\GSOrth(\lpow{2}{\VV}) \ar@{->>}[d] \ar[r] & 1 \\
&& \P(\VV) \ar@{.>}[r]^(.45){\Phi_{\VV}} & \LGrass(\lpow{2}{\VV})^{\circ}&
}
$$
of $X$-scheme, where the top two rows are short exact sequences
of group schemes.

So far, we've shown that $\Phi_{\VV}$ is an isomorphism onto a
connected component.  The same argument can be used for $\Phi_{\VV}'$
(using the dual parabolic subgroup of $\Group{P}$).  Since
$\LGrass(\lpow{2}{\VV}) \to X$ has two connected $X$-components (since
the discriminant of $\lpow{2}{\VV}$ is trivial) and $\Phi_{\VV}$ and
$\Phi_{\VV}'$ map to different components, we are done.
\end{proof}

\begin{corollary}
\label{cor:lineiffmetabolic}
Let $X$ be a scheme with 2 invertible and $\VV$ a vector bundle of
rank 4 on $X$.  Then $\VV$ has a sub- or quotient line bundle if and
only if $\lpow{2}{\VV}$ is metabolic.
\end{corollary}

Of course, if $r=2s$ is even and $\lpow{s}{\VV}$ is metabolic, then
the Witt class $\midext{\VV} \in W^r(X,\det\VV)$ vanishes (see
Remark~\ref{rem:dual_or_not}).  In general, the converse --- that if
$\lpow{s}{\VV}$ is stably metabolic (i.e.\ it's Witt class
$\midext{\VV}$ is trivial) then it is metabolic --- may not hold.
However, under certain hypotheses on $X$, which we shall now outline,
the converse does indeed hold.

\subsection{Stably metabolic forms}
\label{subsec:cancellation}

We introduce the following properties
of an exact category $\category{C}=(\category{C},{}^{\sharp},\varpi)$
with duality and 2 invertible, see
Quebbemann--Scharlau--Schulte~\cite{quebbemann_scharlau_et_al}, \cite{quebbemann_scharlau_schulte},
Knus~\cite[II]{knus:quadratic_hermitian_forms},
Balmer~\cite[\S1.1.2]{balmer:handbook}, or
Walter~\cite[\S6]{walter:GW_groups_triangulated_categories} for
definitions.

\begin{definition}
\label{rem:cancellation}
The \linedef{stable metabolicity} (sM) property:\ 
given $\epsilon$-symmetric objects $(\EE,b)$ and $(\MM,h)$
of $\category{C}$ with $(\MM,h)$ metabolic, if $(\EE,b) \perp (\MM,h)$
is metabolic then $(\EE,b)$ is metabolic.  Equivalently, $(\EE,b)$ is
metabolic if it has trivial Witt class.

The \linedef{Witt cancellation} (Wc) property:\ 
given $\epsilon$-symmetric objects
$(\EE_1,b_1)$, $(\EE_2,b_2)$, and $(\FF,b)$ of $\category{C}$, if
$(\EE_1,b_1)\perp(\FF,b) \isom (\EE_2,b_2) \perp (\FF,b)$
then $(\EE_1,b_1) \isom (\EE_2,b_2)$.

The \linedef{stable Grothendieck--Witt equivalence} (sGW) property:\ 
given $\epsilon$-symmetric objects $(\EE_1,b_1)$,
$(\EE_2,b_2)$, 
$(\MM_1,h_1)$, $(\MM_2,h_2)$, and $(\FF,b)$  of $\category{C}$ with
$(\MM_1,h_1)$ and $(\MM_2,h_2)$ metabolic with isomorphic lagrangians, if
$
(\EE_1,b_1) \perp (\MM_1,h_1) \perp (\FF,b) \isom (\EE_2,b_2) \perp
(\MM_2,h_2) \perp (\FF,b)
$ 
then $(\EE_1,b_1) \isom (\EE_2,b_2)$.
Equivalently, $(\EE_1,b_1) \isom (\EE_2,b_2)$  if they
have the same  Grothendieck--Witt class.

A scheme $X$ is said to satisfy a property if for every line
bundle $\LL$ and every $\epsilon \in \{\pm 1\}$, the property is
satisfied for $\epsilon$-symmetric objects of the exact category
$\VB(X)$ with the duality given by
$\HHom(-,\LL)$.
\end{definition}

\begin{corollary}
\label{cor:with_witt_cancellation}
Let $X$ be a scheme with 2 invertible satisfying (sM). Then a vector
bundle $\VV$ of rank 4 has a sub- or quotient line bundle if and only
if $\midext{\VV} = 0$ in $W^4(X,\det\VV)$, equivalently, $\euler{\VV}
= 0$ in $W^4(X,\det\VV\dual)$.
\end{corollary}
\begin{proof}
The claim concerning $\midext{\VV}$ is a direct consequence of
Corollary~\ref{cor:lineiffmetabolic} and the 
property (M), under which the form $\lpow{2}{\VV}$ is metabolic if and only
if the class $\midext{\VV} \in W^4(X,\det\VV)$ vanishes.  For the
final equivalence, note that by
Proposition~\ref{prop:euler_properties}, we have
$\euler{\VV} = \langle -1 \rangle \tensor \midext{\VV\dual}$ in
$W^4(X,\det\VV\dual)$.  Thus $\euler{\VV}$ vanishes in
$W^4(X,\det\VV\dual)$ if and only if $\midext{\VV\dual}$ does (if and
only if $\midext{\VV}$ does, see Remark \ref{rem:dual_or_not}).
\end{proof}

We will spend the rest of this section exploring these
properties.  First, note that we have the implication (sGW) $\Rightarrow$ (Wc).

\begin{proposition}
\label{prop:affine}
Over a noetherian affine scheme $X$ with 2 invertible, we have the
following implications
(Wc) $\Rightarrow$ (sGW) $\Rightarrow$ (sM).
\end{proposition}
\begin{proof}
We fix, and then suppress the dependence on, some $\epsilon \in \{\pm
1\}$.  Over an affine scheme with 2 invertible, metabolic is
equivalent to hyperbolic; see \cite[\S3~Prop.~1,~Cor.~1]{knebusch}.
In particular, metabolic forms having isomorphic lagrangians are
isometric.  Hence (Wc) $\Rightarrow$ (sGW).  Now, if $(\EE,b)$ is
stably metabolic, then there exist vector bundles $\NN_1$ and $\NN_2$
such that $(\EE,b) \perp H(\NN_1) \isom H(\NN_2)$.  Since $X$ is
affine, there exists a vector bundle $\NN_3$ such that $\NN_1 \oplus
\NN_3 \isom \OO_X^n$ is free.  Thus $(\EE,b) \perp H(\OO_X^n) \isom
H(\NN)$, where $\NN = \NN_1 \oplus \NN_3$.  Inspired by Roy~\cite[Ex.\
7.3]{roy:cancellation}, 
we choose a hyperbolic pair $e$ and $f$ of $H(\OO_X)$ and consider the
image $v + \vp$ in $H(\NN)$ of $e+f$.  Then $\vp(v) = f(e) = 1$, so
that $v$ is a unimodular element generating a free direct summand of
$\NN$.  In particular, $H(\NN) \isom H(\OO_X) \oplus H(\NN')$.  By
(Wc), we can cancel $H(\OO_X)$.  By induction, we deduce that
$(\EE,b)$ is hyperbolic, hence (sM) holds.
\end{proof}

There may exist stably metabolic objects that are not metabolic.
Such an example due to Ojanguren (see
\cite[Ex.~40]{balmer:handbook}): the reduced norm form associated to
the endomorphism algebra of the tangent bundle of the real 2-sphere
$\Spec \R[x,y,z]/(x^2+y^2+z^2-1)$.

\begin{definition}
We say that an exact category $\category{C}$ satisfies the
\linedef{Krull--Schmidt property} (KS) if every object has a unique
(up to permutation) coproduct decomposition into indecomposable
objects, see Atiyah~\cite[Thm.~3]{atiyah:krull_schmidt}.
Additionally, we say that $\category{C}$ satisfies the \linedef{strong
Krull--Schmidt property} (KS+) if for every object $\EE$ the ring
$\End_{\category{C}}(\EE)$ is complete with respect to the
$\rad(\End_{\category{C}}(\EE))$-adic topology, cf.\
\cite[\S3(iii)]{quebbemann_scharlau_schulte} or
\cite[II~\S6.3]{knus:quadratic_hermitian_forms}.
\end{definition}

For example, (KS+) is satisfied in the following cases: the exact category $\VB(X)$ of vector bundles over a
scheme $X$ proper over a field (see~\cite[Cor.~Lemma 10]{atiyah:krull_schmidt} and~\cite[II \S7]{knus:quadratic_hermitian_forms} or
\cite[\S4.1]{quebbemann_scharlau_et_al}), and the category of
projective modules over a (possibly noncommutative) artinian ring
(see~\cite[II Ex.~6.6.2]{knus:quadratic_hermitian_forms}).

It is known that (KS+) $\Rightarrow$ (Wc), see
\cite[Satz~3.4(iii)]{quebbemann_scharlau_et_al},
\cite[\S3.4(1)]{quebbemann_scharlau_schulte}, or
\cite[II~Thm.~6.6.1]{knus:quadratic_hermitian_forms}.  In particular,
the spectrum of an artinian ring satisfies (M) by Proposition~\ref{prop:affine}.

Let $\Sigma$ be a set of indecomposable objects of an exact category
$\category{C}$. An object $\EE$ is of \linedef{type} $\Sigma$ (resp.\
$\Sigma'$) if every indecomposable direct summand of $\EE$ is
isomorphic to some element of $\Sigma$ (resp.\ if no indecomposable
direct summand is isomorphic to some element of $\Sigma$). An important consequences of the strong Krull--Schmidt property is
the Krull--Schmidt theorem for symmetric objects, see
\cite[Thm.~3.2--3.3(1)]{quebbemann_scharlau_schulte} or
\cite[II~Thm.~6.3.1]{knus:quadratic_hermitian_forms}.

\begin{theorem}[Symmetric Krull--Schmidt]
Let $(\category{C},{}^{\sharp},\varpi)$ be an exact category with
duality.  If $\category{C}$ satisfies (KS) then every
$\epsilon$-symmetric object $(\EE,b)$ has an orthogonal decomposition
$$
(\EE,b) \isom \bigperp_{i=1}^{r} (\EE_i,b_i)
$$
with $\EE_i$ of type $\{\NN_i,\NN_i^{\sharp}\}$ and $\NN_i \oplus
\NN_i^{\sharp} \not\isom \NN_j \oplus \NN_j^{\sharp}$ for $i\neq j$.
Moreover, if $\category{C}$ satisfies (KS+), then
this orthogonal decomposition is unique up to isometry and permutation.
\end{theorem}

\begin{remark}
Two metabolic objects having isomorphic lagrangians need not share
Krull--Schmidt decompositions.  For example, over $\P^1$, 
$\OO(-1)^2$
can be presented as a lagrangian of the hyperbolic form $H(\OO^2)$, hence $H(\OO^2)$ and $H(\OO(-1)^2)$ have isomorphic
lagrangians, yet have very different Krull--Schmidt decompositions.
\end{remark}

We recall the following generalization of a
result of Grothendieck \cite{grothendieck:fibres}, cf.\
\cite[\S3.4(3)]{quebbemann_scharlau_schulte} or
\cite[II~Prop.~7.1.1]{knus:quadratic_hermitian_forms}.

\begin{theorem} 
\label{thm:grothendieck}
Let $k$ be an algebraically closed field and
$(\category{C},{}^{\sharp},\varpi)$ an exact $k$-category with
duality satisfying (KS+). Then
$\epsilon$-symmetric objects $(\EE_1,b_1)$ and $(\EE_2,b_2)$ are
isometric if and only if $\EE_1 \isom \EE_2$.
\end{theorem}

If an $\epsilon$-symmetric object $(\EE,b)$ has type $\{\NN,\NN^{\sharp}\}$, with $\NN$
indecomposable and $\NN \not\isom \NN^{\sharp}$, then in fact $(\EE,b) \isom
H(\NN^r)$ is hyperbolic (where $r$ is such that $\EE \isom \NN^r \oplus \NN^{\sharp r}$), see \cite[Thm.~3.3(3)]{quebbemann_scharlau_schulte} or
\cite[II~Prop.~6.4.1]{knus:quadratic_hermitian_forms}.  Otherwise, $(\EE,b)$ is \linedef{$\NN$-isotypic} if it has type
$\{\NN\}$ with $\NN$ indecomposable (then the \linedef{$\NN$-rank} is  well-defined).

\begin{corollary}
\label{cor:grothendieck}
Retain the hypotheses of Theorem~\ref{thm:grothendieck} and assume
that $k$ has characteristic $\neq 2$.  Let $\NN$ be
an indecomposable object and $(\EE,b)$ an $\NN$-isotypic
$\epsilon$-symmetric object.  If $\NN$ has an $\epsilon$-symmetric
structure then $(\EE,b) \isom H(\NN^m)$ or
$(\EE,b) \isom H(\NN^m)\perp \NN$, depending on whether the $\NN$-rank
of $\EE$ is $2m$ or $2m+1$, respectively.  Furthermore, if $\NN$ is
self-dual but has no $\epsilon$-symmetric structure, then
$(\EE,b)$ is hyperbolic.
\end{corollary}
\begin{proof}
Let $\EE$ have $\NN$-rank $n$.  If $\NN$ has an $\epsilon$-symmetric
structure, then $(\EE,b) \isom \NN^{\perp n}$ by
Theorem~\ref{thm:grothendieck}.  Thus $(\EE,b)$ has a symmetric summand
$\NN^{\perp 2m}$,
where $n=2m$ or $n=2m+1$.  But $\NN^{2m} \isom \NN^m \oplus
(\NN^{\sharp})^m$ and hence, again by Theorem \ref{thm:grothendieck},
there's an isometry $\NN^{\perp 2m} \isom H(\NN^m)$.  If $\NN$ has no
$\epsilon$-symmetric structure, the statement follows by the ``reduction
theorem'' technique \cite[Thm.~3.3(3)]{quebbemann_scharlau_schulte}.
\end{proof}

Finally, we show that under the hypotheses of
Theorem~\ref{thm:grothendieck}, a proof of (sM) can be broken into two
components.

\begin{proposition}
\label{prop:stably}
Let $k$ be an algebraically closed field of characteristic $\neq 2$
and $(\category{C},{}^{\sharp},\varpi)$ an exact
$k$-category with duality satisfying (KS+).  Assume that:
\begin{enumerate}
\item If $\NN$ is a self-dual indecomposable object then any
$\NN$-isotypic part of a metabolic object is metabolic.

\item If $\NN$ is an indecomposable $\epsilon$-symmetric object such
that $\NN \perp H(\NN^m)$ is metabolic then $\NN$ is metabolic.
\end{enumerate}
Then $(\category{C},{}^{\sharp},\varpi)$ satisfies (sM).
\end{proposition}
\begin{proof}
Let $(\EE,b)$ be an $\epsilon$-symmetric stably metabolic object and
write $(\EE,b) \perp (\MM_1,b_1) \isom (\MM_2,b_2)$ for metabolic
objects $(\MM_1,b_1)$ and $(\MM_2,b_2)$.  
If every element of the symmetric Krull--Schmidt decomposition of
$(\EE,b)$ has type $\{\NN,\NN^{\sharp}\}$ with $\NN \not\isom
\NN^{\sharp}$, then already $(\EE,b)$ is hyperbolic.  Otherwise,
choose an $\NN$-isotypic part.
Using assumption (1), we can reduce to the
case where $(\EE,b)$, $(\MM_1,b_1)$, and $(\MM_2,b_2)$ are all
$\NN$-isotypic.

First, if $\NN$ has no structure of $\epsilon$-symmetric object, then
$(\EE,b)$ is already hyperbolic by Corollary~\ref{cor:grothendieck}.
Thus we may assume that $\NN$ has a structure of $\epsilon$-symmetric
object.  If $\EE$ has even $\NN$-rank, the by
Corollary~\ref{cor:grothendieck}, $(\EE,b)$ is already hyperbolic.  If
$\EE$ as odd $\NN$-rank, then either $\MM_1$ or $\MM_2$ has odd
$\NN$-rank. But then Corollary~\ref{cor:grothendieck} together with
assumption (2) implies that $\NN$ is metabolic and thus $(\EE,b)$ is
metabolic.  
\end{proof}

Using Proposition \ref{prop:stably}, we can verify (sM) in some cases.
For example, if $X=\P^1$ over an algebraically closed field of
characteristic $\neq 2$, then indecomposable vector bundles are just
line bundles (see \cite{grothendieck:fibres}).  For a given duality
$\HHom(-,\LL)$ on $\VB(X)$, there is at most one self-dual
indecomposable line bundle, namely $\NN=\OO(d)$ if $\LL\isom\OO(2d)$ is
even.  In particular, an $\NN$-isotypic part is hyperbolic if and only
if it has even rank.  Thus, both hypotheses (1) and (2) of
Proposition~\ref{prop:stably} are verified over $\P^1$.

We wonder if (sM) holds for any scheme proper over a field of
characteristic $\neq 2$.

\subsection{On the vanishing of the Euler class}
\label{subsec:Euler_classes}

In this section, we investigate necessary conditions for the vanishing
of the Grothendieck--Witt-theoretic Euler
class  $\euler{\VV} \in GW^r(X,\det\VV\dual)$ of a vector bundle $\VV$ of
rank $r$ on a scheme $X$ with 2 invertible. The following
important special case of Corollary~\ref{cor:euler_metabolic} gives a
sufficient condition for the vanishing of the Euler class.

\begin{proposition}[Fasel--Srinivas~{\cite[Prop.~22]{fasel_srinivas:chow-witt}}]
\label{prop:euler_vanishing}
Let $X$ be a scheme with 2 invertible and $\VV$ a vector bundle of
rank $r$ on $X$.  If $\VV$ has a free sub- or quotient line bundle then
$\euler{\VV} = 0$ in $GW^r(X,\det\VV\dual)$.
\end{proposition}

One may ask when the existence of a free sub- or quotient line bundle
is equivalent to the vanishing of the Grothendieck--Witt-theoretic
Euler class of a vector bundle.  For vector bundles of rank one, this question always has a positive
answer; the following simple argument was communicated to the author
by J.\ Fasel.

\begin{proposition}
Let $X$ be a scheme with 2 invertible and $\LL$ be a line bundle on $X$.
Then $\LL\isom \OO_X$ if and only if $\euler{\LL} =0$ in
$GW^1(X,\det\LL\dual)$.
\end{proposition}
\begin{proof}
By Proposition~\ref{prop:euler_properties} part 1, we have
$\euler{\LL} = H_{\LL\dual}(\OO_X) \in GW^1(X,\LL\dual)$.  In
particular, $\euler{\LL}$ maps to $[\OO_X]-[\LL\dual]$ under the
forgetful map $GW^1(X,\LL\dual) \to K_0(X)$.  This, in turn, maps to
the isomorphism class of $\LL$ under the determinant homomorphism
$\det : K_0(X) \to \Pic(X)$.  Thus if $\euler{\LL}=0$ then $\LL$ is
trivial.
\end{proof}

There exist vector bundles of rank 2 with no sub- or quotient line
bundle but with vanishing Euler class:\ such examples arise as stably
metabolic but nonmetabolic skew-symmetric forms of rank 2 and trivial
determinant.  However, if the stable metabolicity property (sM) is
satisfied (see Definition \ref{rem:cancellation}), this obstruction
should disappear.

\begin{question}
\label{question:2}
Let $X$ be a scheme with 2 invertible and satisfying (KS+) and (sM).
Let $\VV$ be a vector bundle of rank 2.  Is the vanishing the the
Euler class $\euler{\VV} \in GW^2(X,\det\VV\dual)$ equivalent to $\VV$
having a free sub- or quotient line bundle?
\end{question}

To give some partial results, we can proceeding analogously as in the proof
of Corollary~\ref{cor:lineiffmetabolic}, first developing the rank 2
version of Theorem~\ref{thm:LGrass_D3} utilizing the exceptional
isomorphism $\Dynkin{A}_1=\Dynkin{C}_1$.

\begin{theorem}
\label{thm:LGrass_2}
Let $X$ be a scheme with 2 invertible and $\VV$ a vector bundle of
rank $2$ on $X$.  Then
$$
\Phi_{\VV} : \P(\VV) \to \LGrass(\lpow{1}{\VV})
$$ 
is an isomorphism of $X$-schemes.
\end{theorem}
\begin{proof}
The proof is similar to, except easier than, that of
Theorem~\ref{thm:LGrass_D3}.  In place of
Proposition~\ref{prop:A3=D3}, we use the fact that the functor
$\lpow{1}{}$ defines an isomorphism of group schemes $\GL(\VV) \isom
\GSp(\lpow{1}{\VV})$ over $X$, inducing the required isomorphism of
projective homogeneous bundles.
\end{proof}

In particular, we have the following analogue of
Corollary~\ref{cor:lineiffmetabolic}.

\begin{corollary}
\label{cor:metabolic_2}
A vector bundle $\VV$ of rank $2$ fits into a short exact sequence $0
\to \NN \to \VV \to \LL \to 0$ for line bundles $\NN$ and $\LL$ if and
only if $\lpow{1}{\VV}$ is metabolic.
\end{corollary}

Finally, under some stability hypotheses, we can answer
Question~\ref{question:2} in the affirmative.

\begin{proposition}
\label{prop:vanishing_2}
Let $X$ be a scheme with 2 invertible and satisfying (sM) and (sGW).
Let $\VV$ be a vector bundle of rank 2 on $X$.  Then $\VV$ has a free
sub- or quotient line bundle if and only if $\euler{\VV} = 0$ in
$GW^2(X,\det\VV\dual)$.
\end{proposition}
\begin{proof}
By Proposition~\ref{prop:euler_properties} part (1), we have that
$$
\euler{\VV} =
\midext{\VV\dual} + H_{\det\VV\dual}(\OO_X)
$$ 
in $GW^2(X,\det\VV\dual)$. If $\euler{\VV} = 0$, we see that
$\midext{\VV\dual} \in GW^2(X,\det\VV\dual)$ is a metabolic class.
Thus $\lpow{1}{\VV}$ is metabolic by (sM). By Corollary~\ref{cor:metabolic_2}, there is an exact sequence $0 \to \NN \to \VV
\to \LL \to 0$. But then by Proposition \ref{prop:nice_rep}, we have
$\euler{\VV} =
H_{\det\VV\dual}([\OO_X] - [\NN\dual])$.  Thus
$H_{\det\VV\dual}(\OO_X)$ and $H_{\det\VV\dual}(\NN\dual)$ have equal
Grothendieck--Witt classes, hence
$
H_{\det\VV\dual}(\OO_X) \isom H_{\det\VV\dual}(\NN\dual)
$ 
by (sGW).
The classification of binary line bundle-valued quadratic forms
(cf.\ \cite[\S5.2]{auel:clifford}) in
terms of norm forms associated to line bundles on \'etale quadratic
covers of $X$ (in this case, the split cover) shows that either
$\NN\dual$ or $\HHom(\NN\dual,\det\VV\dual) \isom \LL\dual$ is
free. Thus $\VV$ has a free sub- or quotient line bundle.
\end{proof}

At the other extreme (i.e.\ when neither (sM) nor (sGW) are expected
to hold), one can consider the case when $X$ is a noetherian affine
scheme (with 2 invertible) of dimension $d$, where vector bundles are
just projective modules of finite rank. The question of whether the
Euler class in Grothendieck-Witt theory is the only obstruction to
``splitting off a free factor'' has a positive answer when $d=2$ (rank
2 modules being the only case not covered by Serre's splitting
theorem, see Remark \ref{rem:serre_splitting}) by
Fasel--Srinivas~\cite[Thm.~3]{fasel_srinivas:chow-witt}, who also
handle the case of projective modules of rank $\geq 3$ when $d=3$. The
case $d=3$ and projective modules of rank 2 with trivial determinant
follows by the discussion surrounding Fasel
\cite[Thm.~4.3]{fasel:sphere}.  In general, one can only hope that the
Euler class is the obstruction to splitting off a free factor for
projective modules of rank $\geq d$ (rank $d$ being the crucial case),
a question asked in~\cite{fasel_srinivas:chow-witt}.

As for vector bundles of rank four, we have the following partial
result under a stability hypothesis.

\begin{proposition}
\label{prop:vanishing_4}
Let $X$ be a scheme with 2 invertible and satisfying (sGW). Let $\VV$
be a vector bundle of rank 4 with $\det\VV \isom \OO_X$. Then $\VV$
has a sub- or quotient line bundle if and only if $\euler{\VV}=0$ in $GW^4(X)$.
\end{proposition}
\begin{proof}
By Proposition~\ref{prop:euler_properties}, we have
$$
\euler{\VV} = \langle -1 \rangle \tensor \midext{\VV\dual} +
H\bigl(\OO_X - \VV\dual \bigr).
$$
Supposing that $\euler{\VV}=0$, the forms
$\lpow{2}{\VV\dual}\perp H(\OO_X)$ and
$H(\VV)$ become equivalent in $GW^4(X)$.
Hence 
$$
\lpow{2}{\VV\dual}\perp H(\OO_X) \isom
H(\VV)
$$
by (sGW).
Inspired by Roy~\cite[Ex.~7.3]{roy:cancellation}, 
we now proceed similarly as in the proof of
 Proposition~\ref{prop:affine}.
Choose a hyperbolic basis of global sections $e$ and $f$ of $H(\OO_X)$
and consider the image $v + \vp$ in $H(\VV)$ of $e+f$.  Then $\vp(v) =
f(e) = 1$, so that $v$ is a unimodular element generating a free
factor of $\VV$.
\end{proof}

We wonder what the optimal result is in the case of vector bundles
$\VV$ of rank 4.  One approach, suggested by the referee, works over
smooth affine schemes $X$ of dimension 4 over a field $k$ of
characteristic $\neq 2$.  The idea is to use the Chow--Witt-theoretic
top Chern class $\tilde{c}_4({\VV}) \in
\widetilde{CH}{}^4(X,\det\VV\dual)$, which is the precise obstruction
to $\VV$ having a free factor by a theorem of Morel
\cite[Thm.~7.14]{morel}.  See \cite{fasel:Chow-Witt} for more details on the top Chern
class in Chow--Witt theory.  The Gersten--Grothendieck--Witt spectral
sequence provides an edge homomorphism
$\widetilde{CH}{}^4(X,\det\VV\dual) \to GW^4(X,\det\VV\dual)$ mapping
$\tilde{c}_4(\VV)$ to $\euler{\VV}$, see
Fasel--Srinivas~\cite[\S4.2]{fasel_srinivas:chow-witt}.  The interest
then lies in the kernel of this edge homomorphism.

\end{document}